\documentclass[12pt]{article}

\usepackage{amsfonts}
\usepackage{amssymb}
\usepackage{amsthm}
\usepackage{amsmath}
\usepackage{graphicx}

\newtheorem{thm}{Theorem}[section]
 \newtheorem{prop}[thm]{Proposition}
 \newtheorem{lem}[thm]{Lemma}
 \newtheorem{cor}[thm]{Corollary}

  \newtheorem{que}{Question}

  \theoremstyle{claim}
  \newtheorem{clm}{Claim}

\theoremstyle{definition}
\newtheorem{exm}[thm]{Example}
\newtheorem{dfn}[thm]{Definition}
 \newtheorem{rmk}[thm]{Remark}

\def \tp {{\rm tp}}
\def \Sem {{\rm Sem}}

\def \Th {{\rm Th}}
\def \cl {{\rm cl}}

 \newcommand{\lems}{\mathrel{\reflectbox{\ensuremath{\mapsto}}}}

\newenvironment{littleproof}{{\em Proof.}}{\hfill$\bullet$}

\def  \phiar {\makebox[2em]{\makebox[.6em]{$\longrightarrow$}\!\!\!\!$^{^{\phi}}$}}
\def  \phimap {\makebox[2em]{\makebox[.6em]{$\longmapsto$}\!\!\!\!$^{^{\phi}}$}}

\def  \psiar {\makebox[2em]{\makebox[.6em]{$\longrightarrow$}\!\!\!\!$^{^{\psi}}$}}
\def  \psimap {\makebox[2em]{\makebox[.6em]{$\longmapsto$}\!\!\!\!$^{^{\psi}}$}}
\def  \psiprimar {\makebox[2em]{\makebox[.6em]{$\longrightarrow$}\!\!\!\!$^{^{\psi'}}$}}

\def  \piar {\makebox[2em]{\makebox[.6em]{$\longrightarrow$}\!\!\!\!$^{^{\varphi}}$}}

\def  \leqmap {\makebox[2em]{\makebox[.6em]{$\longmapsto$}\!\!\!\!$^{^{\leq}}$}}

\begin{document}

\title{Semi-isolation and  the strict order property}

\author{Sergey Sudoplatov\thanks{Supported by RFBR (grant 12-01-00460-a)}\\Sobolev Institute of Mathematics, Novosibirsk \\
\and Predrag Tanovi\'c\thanks{Supported by the Ministry of Science
and Technology of Serbia}\\Mathematical Institute SANU, Belgrade}

\maketitle

\begin{abstract}
We study semi-isolation as a binary relation on the locus of a
complete type and prove that under some additional assumptions it induces the strict order property.
\end{abstract}

Throughout the    paper   $T$ is a fixed,  complete, first-order
theory in a countable language   and    $M$ is its (infinite)
monster model. $T$ is an  \emph{Ehrenfeucht theory} if it has
finitely many, but more than one,  countable models. The class of Ehrenfeucht theories is quite interesting. There are
numerous results and large bibliography in this area, see
\cite{BSV, S2} for references. The first
example was found by Ehrenfeucht in \cite{V}:
$T_E=\Th(\mathbb Q,<,n)_{n\in \omega}$. It eliminates quantifiers
and has three countable models: the prime model, the saturated
model, and the model prime over a realization of a nonisolated
type. $T_E$ is also a \emph{binary theory}: every formula is equivalent
modulo $T_E$ to a Boolean combination of formulas with at most two
free variables. Not all Ehrenfeucht theories are binary:
non-binary examples can be found in \cite{Per} and \cite{W1}. The
motivating question for our work is:

\begin{que}\label{Q1}
Is there a binary, Ehrenfeucht theory without the strict order
property? In particular, is there such a theory with 3 countable
models?
\end{que}

An important relation in any Ehrenfeucht theory is  semi-isolation
as a binary relation on the locus of a powerful type $p\in
S(\emptyset)$ in a model of $T$ (all these notions are defined in
Section \ref{preli}). There the semi-isolation relation  is either
empty (if $p$ is omitted) or   a $\bigvee$-definable
quasi-order with no maximal elements. If in addition  $T$  has
precisely 3 countable models then the isomorphism type of any
countable model $N$  can be described by combinatorial properties
of the quasi-order:

\begin{enumerate}
\item $N$  is prime iff $p(N)=\emptyset$;

\item $N$ is prime over a realization of $p$ iff there is a
minimal, with respect to  semi-isolation,  element in $p(N)$. In
this case $N$ is prime over any minimal element;

\item $N$ is saturated iff $p(N)\neq\emptyset$ has no minimal
elements.
\end{enumerate}

We note that in  Ehrenfeucht's example the type
$\{n<x\,|\,n\in\omega\}$ determines a complete 1-type $p$ on whose
locus, in any countable model, the semi-isolation (defined later
and denoted by $SI_p$) coincides with $\leq$\,. In particular,
semi-isolation is a relatively definable relation on the locus of
$p$. The strict order property in this example is induced by the
semi-isolation and it is natural to examine whether this will
happen in any binary Ehrenfeucht theory.

\smallskip
One result in this direction    was obtained by Woodrow in
\cite{W}. He proved that if a theory in the language of the
Ehrenfeucht's example eliminates quantifiers and has 3 countable
models then it is quite similar to the  original one; in
particular,  semi-isolation is a relatively definable ordering on the locus of a powerful type. Ikeda,
Pillay and Tsuboi proved that the same happens in the case of an
almost $\aleph_0$-categorical theory with 3 countable models, see
\cite{IPT}. Another result in this direction was obtained by Pillay  in \cite{P1}
who proved that in any Ehrenfeucht  theory with few links there exists a definable  linear ordering. The ordering relation that he found, when restricted to the locus of a powerful type, is induced by the semi-isolation relation.

\smallskip
In this article we will investigate proper quasi-orders of the
form $(p(M), SI_p)$, where $p\in S(\emptyset)$ is a nonisolated type in an arbitrary first-order theory and  prove that under some additional assumptions   a relatively definable sub-order can be found. The additional assumptions have topological flavour. That is not surprising because  $SI_p$ has a natural topological "definition", the set $S^p_{\rightarrow}$. More precisely, we will consider the set $S_{p,p}$ of all complete extensions of $p(x)\cup p(y)$; it is compact and corresponds to set of all pairs of realizations of $p$. Similarly, $SI_p$ corresponds to the set $S^p_{\rightarrow}$ of all types $\tp(a,b)$ where $(a,b)\in SI_p$. We will decompose $S_{p,p}$ into four   parts, adequate for studying definability properties of $SI_p$ (see Definition \ref{DS} and Remark \ref{RS}). Then we will translate definability properties of semi-isolation  into topological (complexity) properties of these parts.

 In Section \ref{S2} we will prove that certain assumptions on the complexity imply the existence of a proper, relatively definable sub-order of $SI_p$.  For example, we will prove in Theorem \ref{Tpd}
that if the theory $T$ has {\em  closed asymmetric links on $p(M)$}  (meaning that one of the parts, the set $S^p_{\mapsto}$,  is  non-empty and closed in $S_{p,p}$) then there exists a non-trivial, relatively definable    sub-order of $SI_p$.  This generalizes Pillay's result in one direction: if   $p$ is a powerful type of an  Ehrenfeucht theory with few links  then $S^p_{\mapsto}$ is finite (hence closed) and non-empty.

In Sections \ref{S3} and \ref{S4} we concentrate on the existence of antichains in $SI_p$ in the case of an NSOP theory. We don't do much in this direction:  assuming that the underlying theory is binary, NSOP and has three countable models, with lots of efforts  we prove that there are at least two distinct types of $SI_p$-incomparable pairs of elements on the locus of a powerful type. This indicates that the answer to Question \ref{Q1} may be affirmative.

In Section \ref{S5} we  consider a powerful  type $p$ in a binary theory  for which $SI_p$ is downwards directed in a specific way (PGPIP). We prove that in the NSOP case  the Cantor-Bendixson rank of $S_{p,p}$ is finite; this indicates that maybe there are no binary, Ehrenfeucht, NSOP theories with PGPIP at all. So the answer to Question \ref{Q1} may be  negative!?

\section{Preliminaries}\label{preli}

\medskip
Throughout the paper  $S_n(A)$ denotes the  set of all complete
$n$-types with parameters from $A$. The topology on $S_n(A)$ is
defined in a usual way. If $\phi(\bar x)$ is a formula over $A$ in $n$ free variables then by $[\phi]$ we will denote the set of all types from $S_n(A)$ containing $\phi(\bar x)$.
$S(A)$ denotes $\bigcup_{n}S_n(A)$. If
$p,q\in S(\emptyset)$ then $S_{p,q}(\emptyset)$ is the subspace of
all the extensions of $p(\bar x)\cup q(\bar y)$ in
$S_m(\emptyset)$ (where $\bar x$ and $\bar y$ are disjoint and
$m=|\bar x|+|\bar y|$). Similarly, if $q\in S_n(\emptyset)$ then
$S_q(A)$ denotes the set of all completions of $q(\bar x)$ in
$S_n(A)$. For any $\bar c$ realizing $p$ there is a canonical
homeomorphism between $S_{p,q}(\emptyset)$ and $S_{q}(\bar c)$:
the one sending $r(\bar x,\bar y)$ to $r(\bar c,\bar y)$.

\smallskip
Next we recall the definition of the Cantor-Bendixson rank. It is
defined on the elements of a topological space $X$ by induction:
$CB_X(p)\geq 0$ for all $p\in X$; \ \ $CB_X(p)\geq\alpha$ iff  for
any $\beta<\alpha$ $p$ is an accumulation point of the points of
$CB_X$-rank at least $\beta$. $CB_X(p)=\alpha$ \ iff \ both
$CB_X(p)\geq\alpha$ and $CB_X(p)\ngeq\alpha+1$ hold; if   such an
ordinal $\alpha$ does not exist then $CB_X(p)=\infty$.   Isolated
points of $X$ are precisely those having rank $0$, points of rank
$1$ are those which are isolated in the subspace of all
non-isolated points, ... For a non-empty $C\subseteq X$  we define
$CB_X(C)=\sup\{CB_X(p)\,|\,p\in C\}$; in this way $CB_X(X)$ is
defined and $CB_X(\{p\})=CB_X(p)$ holds.  If $X$ is compact and
$C$ is closed in $X$ then the sup is achieved: $CB_X(C)$ is the
maximum value of $CB_X(p)$ for $p\in C$; there are  finitely many
points of maximum rank in $C$ and the number of such points is the
\emph{$CB_X$-degree} of $C$. If $X$ is countable and compact then
$CB_X(X)$ is a countable ordinal and every  closed subset has
ordinal-valued rank and finite $CB_X$-degree.

$S_n(A)$ is compact so  $CB$-rank is defined   there on points
(complete types) and well behaves on closed subsets (they
correspond to partial types). So whenever $p$ is a partial type in
$n$ free variables   and parameters from $A$ then $CB^A_n(p)$
 is the $CB$-rank  of the compact space consisting  of
all completions of $p$ in $S_n(A)$; usually the meaning of $n$ and
$A$ will be clear from the context so we will simply write
$CB(p)$. Similarly the $CB$-degree is defined. Thus the $CB$-rank
and degree are defined on all partial types and,  in particular,
they are defined  on formulas. If $T$ is small then the value of
the $CB$-rank of a partial type over a finite domain is an
ordinal.

\smallskip
$\phi(M,\bar a)$ denotes the solution set of $\phi(\bar x,\bar
a)$; if $p(\bar x)$ is a (partial) type then by $p(M)$ we denote
the set of all its realizations. $D\subseteq M^n$ is definable if
it is defined by a formula with parameters; it is $A$-definable
(or definable over $A$) if the defining formula can be chosen
to use only parameters from $A$. $D$ is type-definable
($\bigvee$-definable) if it is  the intersection (union) of $<|M|$
definable sets; if all the sets in the intersection (union) are
definable over a fixed set $A\subset M$ then $D$ is type-definable
($\bigvee$-definable)  over $A$. In this paper we will consider
only   countable intersections and unions  of sets definable over
a finite parameter set. Let $C\subseteq M^n$ be type-definable and
let $C_1\subseteq C$. $C_1$ is \emph{relatively definable within
$C$} if there is a definable $D\subseteq M$ such that $C_1=C\cap D$;
similarly relative $\bigvee$-definability is defined.

\smallskip

 Semi-isolation  was introduced by Pillay in \cite{P1};
here we will sketch its basic properties and   more details the
reader can find in \cite{BSV}.  $\bar b$ is {\em semi-isolated
over} $\bar a$ (or $\bar a$ {\em semi-isolates $\bar b$})\, iff
\,there is a formula $\phi(\bar a,x)\in \tp(\bar b/\bar a)$ such
that $\phi(\bar a,x)\vdash \tp(\bar b)$; we will denote this by
$\bar b\in \Sem(\bar a)$, or by $\bar a\longrightarrow \bar b$.
$\phi(\bar x,\bar y)$ is said to witness the semi-isolation, we
will also write $\bar a\phiar \bar b$ ($\bar a$ $\phi$-arrows
$\bar b$). Thus:
\begin{center}
$\bar a\phiar \bar b$ \ \ if and only if \ \  $\models\phi(\bar
a,\bar b)$ and $\phi(\bar a,\bar y)\vdash \tp_{\bar y}(\bar b)$.
\end{center}
If $\bar a\longrightarrow \bar b$ then  there are many formulas
witnessing the semi-isolation: if $\phi(\bar x,\bar y)$ is a
witness    then $\phi(\bar x,\bar y)\wedge \bar x=\bar x$ is a
witness, too. Therefore we can have many distinct named arrows
between a fixed pair of tuples.

The reader may note that our definition of  $\bar a\longrightarrow
\bar b$ does not exclude the existence of an arrow in the opposite
direction. If, in addition to $\bar a\longrightarrow \bar b$, we
know that the opposite arrow does not exist (i.e. that $a\notin\Sem(b)$) we will write $\bar
a\longmapsto \bar b$. Therefore $\bar a\longmapsto \bar b$ means
that  both $\bar a\longrightarrow \bar b$ and $\bar
a\notin\Sem(\bar b)$ hold; $\bar a\longrightarrow \bar b$ and
$\bar a\longmapsto \bar b$ may be  consistent.
$\bar a\lems\bar b$ means $\bar b\longmapsto \bar a$. Finally,
$\bar a\longleftrightarrow \bar b$ means that both $\bar
a\longrightarrow \bar b$ and $\bar b\longrightarrow \bar a$ hold.

\smallskip  Consider semi-isolation as a binary relation on $M^{<\omega}$.
It is trivially reflexive
and  it is not hard to see that it is transitive:  \begin{center}
 $\bar
a\phiar\bar b$  \  and  \ $\bar b\psiar\bar c$ together  \ \ imply
\ \ $\bar a\piar\bar c$;\end{center}  where $\varphi(\bar x,\bar
z)$ is $\exists\bar y (\phi(\bar x,\bar y)\wedge \psi(\bar y,\bar
z))$. Thus semi-isolation is a quasi-order on $M^{<\omega}$. We note an interesting  consequence of transitivity:
$$\bar a\longmapsto \bar b\longrightarrow \bar c \textmd{ \ \ implies \ \ } \bar a\longmapsto \bar c \ .$$
We shall be  interested mainly in semi-isolation as a binary relation on the locus of a complete type $p\in S(\emptyset)$. Then it is relatively $\bigvee$-definable
within the locus:  to simplify notation we will consider
only 1-types, this is justified by passing to an appropriate sort
in $M^{eq}$. So fix for a while $p\in S_1(\emptyset)$. Define
$$SI_p=\{(a,b)\in p(M)^2\,|\,
a\longrightarrow b\} $$ For any $(a,b)\in SI_p$ there exists an $L$-formula
$\phi(x,y)$ witnessing $p$-semi-isolation. This implies that $SI_p$ is defined by
$\bigvee\phi(x,y)$ within $p(M)^2$ (here the disjunction is taken over all such
$\phi$'s), so $SI_p$ is a relatively
$\bigvee$-definable subset of $p(M)^2$.

 Define:
$$\overline{SI}_p=\{(a,b)\in p(M)^2\,|\,
a\longrightarrow b \textmd{ or } b\longrightarrow a \textmd{ holds
}\}\,; \ \ \ \ \ \perp_p=p(M)^2\smallsetminus \overline{SI}_p
\,.$$ $(a,b)\in\perp_p$ means that $a,b$ are incomparable in the
quasi-order, in which case we will write $a\perp_p b$.
$\overline{SI}_p$ is  relatively $\bigvee$-definable within
$p(M)^2$, while $\perp_p$ is type-definable.

\smallskip We  shall use the following syntax: $x\notin\Sem_p(y)$  will denote the type consisting
of all  negated  formulas witnessing $p$-semi-isolation;
$x\perp^p y$  will denote the type $x\notin \Sem_p(y)\cup y\notin
\Sem_p(x)$. Therefore the type $p(x)\cup p(y) \cup x\perp^p y$
defines the set $\{(a,b)\in p(M)^2\,|\,a\perp_p b\}$ whose
complement in $p(M)^2$ is $\overline{SI}_p$.

\smallskip
Each   $\phi(x,y)$ witnessing $p$-semi-isolation  defines a binary relation on $p(M)$, so the quasi-order $SI_p$ may
also be viewed as the union of a family of binary relations; this has already been suggested by   the
arrows-notation. The relations defined by arrows correspond naturally to subsets of $S_{p,p}$ and relative definability properties translate into topological properties of these subsets.

\begin{dfn}\label{DS}For a non-isolated    $p\in S(\emptyset)$ and $\sigma\in\{\mapsto,\lems,
\rightarrow,\leftarrow,\leftrightarrow,\perp\}$ define:
\begin{center}
$S^p_{\sigma}=\{\tp(ab)\in S_{p,p}\,|\,a\ \sigma\ b\}$
\end{center}
\end{dfn}

The non-isolation of $p$ in the definition is assumed in order to exclude the trivial case  $SI_p=p(M)^2$, which is not interesting at all.

\begin{rmk}\label{RS}Let $p\in S(\emptyset)$ be non-isolated. We list some observations related to the defined parts of $S_{p,p}$:

\smallskip
(1) \  $S_{p,p}$ is the disjoint union: \
$S_{p,p}=S^p_{\mapsto}\cup S^p_{\lems}\cup S^p_{\perp}\cup
S^p_{\leftrightarrow}$ .

\smallskip
(2) \  The mapping taking $\tp(a,b)$ to $\tp(b,a)$ is a homeomorphism of $S_{p,p}$. It fixes setwise
$S^p_{\perp}$ and $S^p_{\leftrightarrow}$, and maps: $S^p_{\mapsto}$ onto $S^p_{\lems}$,  and $S^p_{\rightarrow}$ onto $S^p_{\leftarrow}$. In particular, $S^p_{\mapsto}$ and $S^p_{\lems}$, as well as $S^p_{\rightarrow}$ and $S^p_{\leftarrow}$ are homeomorphic.

\smallskip

(3) \   $S^p_{\leftrightarrow}$   has at least one member
(containing $x=y$). Each of  $S^p_{\mapsto}, S^p_{\lems}$ and $S^p_{\perp}$
may be empty while their union is non-empty. By part (2) $S^p_{\mapsto}$ and $S^p_{\lems}$ are homeomorphic, so  they are either both empty or both non-empty.
\begin{itemize}
\item Consider  the theory of an infinite set with infinitely many elements
named and  let $p\in S_1(\emptyset)$ be the unique non-algebraic
type. Then  $S^p_{\mapsto}=S^p_{\lems}=\emptyset$,
while $S^p_{\perp}$ is a singleton with a member containing $x\ne
y$.

\item Consider  the type $p\in
S_1(\emptyset)$ containing $\{n<x\,|\,n\in\omega\}$ in
Ehrenfeucht's theory $T_E$. There $S^p_{\mapsto}$ and
$S^p_{\lems}$ have members containing $x<y$ and $y<x$ respectively,
while $S^p_{\perp}=\emptyset$ because any two elements are comparable.
\end{itemize}

\smallskip
(4) \  $ S^p_{\mapsto}\cup S^p_{\leftrightarrow}= S^p_{\rightarrow}$ \ \ and \ \  $S^p_{\lems}\cup S^p_{\leftrightarrow}= S^p_{\leftarrow}$

\smallskip
(5) \     $S^p_{\rightarrow}$,  $S^p_{\leftarrow}$ and  $S^p_{\leftrightarrow}$ are open in $S_{p,p}$:    $S^p_{\rightarrow}$ is open because   $S^p_{\rightarrow}=\bigcup_{\phi}[\phi]$ where the union is taken over all formulas $\phi(x,y)$ witnessing $p$=semi-isolation; by homeomorphism      $S^p_{\rightarrow}$ is   open, too. If $\tp(a,b)\in S^p_{\leftrightarrow}$ then there is a formula $\varphi(x,y)\in \tp(a,b)$ witnessing $a\longleftrightarrow b$ and $S^p_{\leftrightarrow}$ is the union $\bigcup_{\varphi}[\varphi]$ taken over all such $\varphi(x,y)$. $S^p_{\leftrightarrow}$  is open in $S_{p,p}$.

\smallskip
(6) \    $S^p_{\perp}$ is closed in $S_{p,p}$ because it is the set of all completions of $p(x)\cup p(y) \cup x\perp^p y$.

\smallskip
(7) \    Since $SI_p$ corresponds to $S^p_{\rightarrow}$   $SI_p$ is relatively definable within $p(M)^2$ iff $S^p_{\rightarrow}$ is clopen in $S_{p,p}$.  But $S^p_{\rightarrow}$ is always open, so  $SI_p$ is relatively definable iff $S^p_{\rightarrow}$ is closed in $S_{p,p}$.

\smallskip
(8) \   $\overline{SI_p}$ corresponds to $S^p_{\rightarrow}\cup S^p_{\leftarrow}$, which is open.  Therefore  relative definability of  $\overline{SI_p}$  within $p(M)^2$ is equivalent to either of the following  conditions:
\begin{itemize}
\item  $S^p_{\rightarrow}\cup S^p_{\leftarrow}$ is clopen in $S_{p,p}$;

\item  $S^p_{\rightarrow}\cup S^p_{\leftarrow}$ is closed in $S_{p,p}$;

\item   $S^p_{\perp}$ is clopen in $S_{p,p}$  (because it is the relative complement of $S^p_{\rightarrow}\cup S^p_{\leftarrow}$).
\end{itemize}

\smallskip
(9) \  $\cl(S^p_{\mapsto})\subseteq S^p_{\mapsto}\cup
S^p_{\perp}$ \ (where $\cl$ denotes the topological closure in $S_{p,p}$). Since $S^p_{\leftarrow}$ is open and disjoint from
 $S^p_{\mapsto}$ we have  $\cl(S^p_{\mapsto})\subseteq
S_{p,p}\smallsetminus S^p_{\leftarrow}=S^p_{\mapsto}\cup
S^p_{\perp}$. In particular, if $S^p_{\mapsto}$ is not closed
then it has an accumulation point in $S^p_{\perp}$ and
$S^p_{\perp}\neq\emptyset$.
\end{rmk}

\begin{dfn}
A non-isolated type  $p\in S(\emptyset)$   is {\em symmetric} iff
$SI_p$ is a symmetric binary relation on $p(M)$. Otherwise, $p$ is
{\em asymmetric}.
\end{dfn}

Since semi-isolation is transitive, it follows   that $P$ is asymmetric if and only if $(p(M),SI_p)$ is a proper quasi-order (with infinite strictly increasing chains).
Asymmetric types  may exist even in an  $\omega$-stable theory  so
their existence, in general, does not imply the strict order
property; examples of that kind can be found in \cite{S1,S2} and \cite{T2}.

\begin{rmk}\label{Rnew0}
It is well known that the symmetry of  semi-isolation implies the
symmetry of isolation. We will sketch the proof of this fact.

\smallskip
(1)  If  $\tp(a/b)$ is isolated and $b\in\Sem(a)$ then  $\tp(b/a)$
is isolated, too. To prove this fact  choose $\phi(x,b)\in\tp(a/b)$
witnessing the isolation  and choose  $\psi(a,y)\in\tp(b/a)$ witnessing
the semi-isolation. Then $\psi(a,y)\wedge\phi(a,y)\vdash
\tp(b/a)$: if $b'$ satisfies this formula then $\models\psi(a,b')$
implies $\tp(b')=\tp(b)$. Combining with $\models\phi(a,b')$ (and $\phi(x,b)\vdash\tp(a/b)$) we derive
$\tp(ab')=\tp(ab)$;   $\tp(b/a)$ is isolated.

\smallskip
(2) Suppose that $\tp(a/b)$ is isolated and that $\tp(b/a)$ is nonisolated. Then $b\longrightarrow a$ and, by part (1), $b\notin \Sem(a)$. This shows that the asymmetry of isolation  on  a pair of elements implies the asymmetry of semi-isolation on the same pair. In particular, if  $p\in S(\emptyset)$ and there are $a,b\models  p$ such that
$\tp(a/b)$ is isolated and  $\tp(b/a)$ is nonisolated, then $p$ is
asymmetric.

\smallskip
(3)  Suppose that $\tp(a/b)$ is isolated. By part (1) we have:
\begin{center}
  $\tp(b/a)$ is nonisolated \ iff \ $b\notin \Sem(a)$ \ iff \ $b\longmapsto a$.
\end{center}
\end{rmk}

The following example shows that the symmetry of semi-isolation
does not necessarily imply the symmetry of isolation on $p(M)$.

\begin{exm}\label{Ex1} Let  $T=\Th(\omega,<)$. Here there is  a unique non-algebraic 1-type
$p(x)$ over $\emptyset$ (the type of an infinite element). Any
infinite element has an immediate successor and a predecessor,  so
$x\pm n$ are well-defined functions and
$$SI_p=\bigcup_{n\in \omega}\{(x,y)\in p(M)^2\,|\, x-n<y\} \ ;$$
(note that $x+n\leq y$ is implied by $x<y$). $p$ is asymmetric: take $a,b$ realizing $p$ such that $a+n<b$ holds for all integers $n$; then $a\longmapsto b$.  On the other hand, isolation on $p(M)$ is symmetric because it is witnessed by a formula of the form $x=y\pm n$ for some $n$.

Note that  $SI_p$ is not relatively definable within $p(M)^2$, because the union is strictly
increasing. On the other hand, $\overline{SI}_p=p(M)^2$ is
obviously relatively definable within $p(M)^2$ so there are asymmetric types for which $\overline{SI_p}$ is   relatively definable although $SI_p$ is not relatively definable within the locus.
\end{exm}

 Recall that a nonisolated type $p\in S(\emptyset)$ is called
{\em powerful} if the model prime over a realization of $p$ is
weakly saturated (realizes all finitary types over $\emptyset$).
Benda in \cite{B} proved that powerful types exist in any
Ehrenfeucht theory: Consider all the (isomorphism types of)
countable models atomic over a finite subset and order them by
elementary embeddability. Then there is a maximal element (since
there are finitely many isomorphism types);   the maximal models
are precisely those that are weakly saturated.

\begin{rmk}\label{R111}
We note some well-known facts about powerful types. For reader's
convenience  we will  sketch  their proofs.

\smallskip
(1) Any  powerful type is asymmetric. Let  $p(x)$ be powerful and
let $a\models p$. Since $p$ is nonisolated  we can find $a'$
realizing a nonisolated extension of $p$ in $S(a)$. Further,
because $\tp(a a')$ is realized in any maximal model, there is
$b\models p$ such that $\tp(a a'/ b)$ is isolated. Note that
$\tp(a'/ab)$ is isolated. If  $\tp(b/a)$ were   isolated then, by
transitivity of isolation,  $\tp(a'b/a)$ would be isolated, too.
The later implies    isolation of  $\tp(a'/a)$; a contradiction.
Therefore $\tp(b/a)$ is nonisolated while   $\tp(a/b)$ is
isolated, so  isolation is asymmetric on $p(M)$. By Remark \ref{Rnew0}(2)  we conclude that $p$ is
asymmetric.

\smallskip
(2) Let $p$ be powerful. Then the proof of part (1) shows that for
any $a\models p$ there exists $b\models p$ such that $b\longmapsto
a$.

\smallskip
(3) Semi-isolation  is
a downwards directed quasi-order on the locus of a powerful type:
If   $a, b$ realize $p$ then, by maximality, there is $d$
realizing $p$ such that $\tp(a b/ d)$ is isolated. In particular,
$\tp(a/ d)$ and $\tp(b/d)$ are isolated, by $\phi(d, x)$ and
$\psi(d,y)$ say and we have $d\phiar a$ and $d\psiar b$. $d$ is a
lower bound for  $a$ and $b$.
\end{rmk}

 By a \emph{$p$-principal formula} we  mean an $L$-formula
$\phi(x,y)$ such that for some (any) $a$ realizing $p$:
\begin{center}
$\phi(a,x)$ isolates an extension of $p$ in $S_1(a)$ and $a\phimap
b$ holds for all $b\in \phi(a,M)$.
\end{center}
By Remark \ref{Rnew0}(3) the condition $a\phimap b$ can be
replaced by `$\tp(b/a)$ is nonisolated'.


\begin{rmk}\label{Rnew2} Suppose that $p$ is powerful.
We  strengthen the conclusion of  Remark \ref{R111}(3):  \
 for all $a,b\in p(M)$ there is
$d\in p(M)$ and $p$-principal formulas $\phi$ and $\psi$ such that
both \ $d\phimap a$ \  and \  $d\psimap b$ hold. To prove it first
choose $c_a,c_b\models p$ satisfying  $c_a\longmapsto a $ and
$c_b\longmapsto  b$ (here we use Remark \ref{R111}(2)). Then
choose $d\models p$ such that $\tp(c_ac_bab/d)$ is isolated. Then
$\tp(c_a/d)$ is isolated, by $\phi(d,x)$ say.  Further,
$d\longrightarrow c_a\longmapsto a$ implies $d\longmapsto a$ and
$d \phimap a$. Similarly, $d\psimap b$ for a suitably chosen
$\psi$.
 \end{rmk}

Recall that a theory $T$ is \emph{binary} if  every formula is
equivalent modulo $T$ to a Boolean combination of formulas with at
most two free variables. Binary theories are a special case of
\emph{$\Delta$-based} theories (\cite{SPS}).  There $\Delta$ is a fixed set of formulas (without parameters)  and
every formula without parameters  is equivalent to a Boolean combination of
formulas from $\Delta$. As noticed in \cite{SPS} this means precisely that any
complete type $p\in S(\emptyset)$ is $\Delta$-based, i.e. that $p$ is forced by the set of
formulas $\phi^{\delta}\in p$, where $\phi\in\Delta$ and
$\delta\in\{0,1\}$. In particular, a theory is binary if and only
if any complete type is forced by the union of its $2$-subtypes.

\section{Definability of semi-isolation}\label{S2}

In this section we study definability properties of semi-isolation on the locus of an asymmetric type $p\in S(\emptyset)$.
We know that $SI_p$ is $\bigvee$-definable within $p(M)^2$.  We will prove that certain  additional assumptions on the topological complexity of $S_{p,p}$ imply the strict order property (SOP). The ordering relation found will always be a subset of $SI_p$, as formalized in the next definition.

\begin{dfn}
Suppose that $p\in S(\emptyset)$ and that $(p(M),\leq)$ is a quasi-order with infinite strictly increasing chains. We will say that $\leq$ is a {\em $p$-order} if:

\smallskip(1) $\leq$ is a relatively definable subset of $p(M)^2$; and

\smallskip(2) $a\leq b$ implies $(a,b)\in SI_p$.
\end{dfn}

The next proposition shows that a $p$-order  is the restriction of a definable quasi-order to $p(M)$; the domain of such a quasi-order can be chosen to be definable and   unbounded (contains no maximal elements).

\begin{prop}\label{Lporder}
Suppose that $p\in S(\emptyset)$,   $(p(M),\leq)$ is a  $p$-order, and that $\varphi(x,y)$ relatively defines $\leq$ within $p(M)^2$. Then there exists $\theta(x)\in p$ such that the formula $\theta(x)\wedge\theta(y)\wedge\varphi(x,y)$ witnesses $p$-semi-isolation and defines an unbounded quasi-order on $\theta(M)$.
\end{prop}
\begin{proof}
Denote by  $\tau(x,y,z)$   the formula \ $\varphi(x,x)\wedge((\varphi(x,y)\wedge\varphi(y,z)\Rightarrow\varphi(x,z))$

\noindent
The first condition from the definition of a $p$-order implies:
\setcounter{equation}{0}
\begin{equation}
p(x)\cup p(y)\cup p(z)\vdash \tau(x,y,z)
\end{equation}
The second can be expressed by:
\begin{equation}
p(x)\cup p(y)\cup\{\varphi(x,y)\}\vdash \bigvee_{i\in I}\phi_i(x,y)
\end{equation}
where the disjunction is taken over all formulae witnessing
$p$-semi-isolation. By compactness there exists a finite
$I_0\subset I$ such that (2) holds with $I_0$ in place of $I$.
Then:
\begin{equation}
p(x)\cup p(y)\cup\{\varphi(x,y)\}\vdash  \phi(x,y)\ ,
\end{equation}
where $\phi(x,y)$ is the formula $\bigvee_{i\in I_0}\phi_i(x,y)$.  Note that $\phi(x,y)$ witnesses $p$-semi-isolation. Now we apply  compactness simultaneously to (1) and (3): there exists a formula $\theta_0(x)$ such that
\begin{equation}
\theta_0(x)\wedge \theta_0(y)\wedge\theta_0(z)\vdash \tau(x,y,z) \ \ \textmd{and} \ \ \theta_0(x)\wedge \theta_0(y)\wedge\varphi(x,y)\vdash \phi(x,y)
\end{equation}
The first relation here implies that $\varphi(x,y)$ defines a quasi-order $\leq_{\varphi}$ on $\theta_0(M)$; its   restriction to $p(M)$ is $\leq$. The second implies that
$\theta_0(x)\wedge \theta_0(y)\wedge\varphi(x,y)$ witnesses $p$-semi-isolation.  Now we show that there is no $\leq_{\varphi}$-maximal element in $\theta_0(M)$ above   $a\in p(M)$.  $a\leq_{\varphi} b$ implies $b\in p(M)$ and, because $\leq$  is a $p$-order, there exists a strictly $\leq$-increasing chains above $b$. Thus $b$ is not $\leq$-maximal. But $\leq$ is a restriction of $\leq_{\varphi}$, so $b$ is not $\leq_{\varphi}$-maximal.

\smallskip
Let $\theta(x)$ be the conjunction of $\theta_0(x)$ and the formula saying that there is no $\leq_{\varphi}$-maximal element above $x$. Clearly, $\theta(x)\wedge \theta(y)\wedge\varphi(x,y)$ witnesses $p$-semi-isolation and defines the restriction of $\leq_{\varphi}$  on $\theta(M)$. To finish the proof it remains to show  that the restricted quasi-order is unbounded;  this holds because  $\theta(M)$ is  $\leq_{\varphi}$-closed upwards in $\theta_0(M)$ and  $\theta_0(M)$ is unbounded.
\end{proof}

As an immediate corollary we obtain:

\begin{cor}\label{SIdef}
If   $p(x)\in S(\emptyset)$ is asymmetric and     $SI_p$ is a
relatively definable subset of $p(M)^2$ then there is $\theta(x)\in p$ and  a definable, unbounded
quasi-order on $\theta(M)$ whose restriction to $p(M)$ is $SI_p$.
In particular, $T$ has the strict order property.
\end{cor}

This  fact is well known and can be found in
different forms in \cite{BSV,IPT,P1} and \cite{T1}. An example of an asymmetric type with relatively definable semi-isolation is the unique non-isolated 1-type in the Ehrenfeucht's example. A similar situation appears in any
almost $\aleph_0$-categorical theory: recall that $T$ is \emph{almost $\aleph_0$-categorical} (see \cite{IPT}) if \
$p_1(x_1)\cup p_2(x_2)\cup...\cup p_n(x_n)$ \ has only finitely
many completions   $r(x_1,...,x_n)\in S(\emptyset)$ for all $n$
and  all complete types $p_i(x_i)\in S(\emptyset)$. For any $p$ in such a theory $SI_p$ is relatively definable within $p(M)^2$: $S_{p,p}$ is finite, so all its the relevant parts are clopen and, by Remark \ref{RS}, $SI_p$ is relatively definable; alternatively: there are
only finitely many inequivalent formulae witnessing
$p$-semi-isolation, so their disjunction relatively defines $SI_p$ within
$p(M)^2$.

\begin{cor}\label{Cfinite}
If   $p(x)\in S(\emptyset)$ is asymmetric and     $S_{p,p}$ is finite
then there is $\theta(x)\in p$ and  a definable, unbounded
quasi-order on $\theta(M)$ whose restriction to $p(M)$ is $SI_p$.
In particular, $T$ has the strict order property.
\end{cor}

\begin{exm} Let $T=(\mathbb{Q},<,c_n,d_n)$ where $(c_n)$ is an increasing and $(d_n)$ is a decreasing sequence such that both  converge to $\sqrt{2}$. $T$ is an Ehrenfeucht theory having 6 countable models. Let $p$ be the 1-type representing  "$\sqrt{2}$". Then the  locus is $p$ is convex and linearly ordered by $<$. However, $p$ is symmetric, and $SI_p$ is the identity relation. Thus  there is no $p$-order there!

\smallskip Therefore,  even the locus of a  symmetric type may be properly ordered, so the asymmetry of semi-isolation is not an exclusive reason for the presence of the strict order property. However, we believe that in this example the reason for the absence of $p$-orders lies in non-powerfulness of $p$.
\end{exm}

\begin{que}Suppose that $p$ is a powerful type in an Ehrenfeucht theory and that $p(M)$ is properly ordered (meaning that there are $a,b$ realizing $p$ such that $a<b$). Must there exist a $p$-order?
\end{que}

It is easy to realize that relative definability of $SI_p$ implies relative definability of $\overline{SI}_p$ within $p(M)^2$. The converse is, in general, not true as Example \ref{Ex1} shows:  there the asymmetric type $p\in S_1(\emptyset)$ is such that   $\overline{SI}_p$ is relatively definable within $p(M)^2$, while $SI_p$ is not so.

We will prove in Corollary \ref{P11} below that relative definability of  $\overline{SI}_p$
for asymmetric $p$ implies the existence of a $p$-order. Actually, the order found in the proof will have an additional property which will witness that semi-isolation is \emph{partially
definable} on $p(M)$. This notion was introduced in \cite{T2} and here we give an equivalent definition which relies on the notion of a $p$-order:

\begin{dfn} We say that semi-isolation is {\em partially definable on $p$} if
there is a definable quasi-order $\leq$ such that for all $a\in
p(M)$:

\smallskip (i) \ \ \  the restriction of $\leq$ to $p(M)$ is a $p$-order, and

\smallskip  (ii) \ \   $a\leqmap b\longrightarrow b'$ and $b'\in p(M)$  imply
$a\leqmap b'$ .
\end{dfn}

    Clearly,  partial  definability of semi-isolation implies that $T$ has the strict order property.

\begin{que}
Does the existence of a $p$-order imply partial definability of semi-isolation on $p$?
\end{que}

\begin{thm}\label{Tpd}  Suppose that $p\in S(\emptyset)$ is asymmetric and that $S^p_{\mapsto}$ is closed in $S_{p,p}$. Then semi-isolation is partially definable on $p(M)$. In particular, $T$ has the strict order property.
\end{thm}
\begin{proof}Suppose that $S^p_{\mapsto}$ is closed in $S_{p,p}$. Then it is compact.
For each $q(x,y)\in S^p_{\mapsto}$ choose a formula $\varphi_q(x,y)\in q(x,y)$ witnessing $p$-semi-isolation. Then $S^p_{\mapsto}\subseteq \bigcup\{[\varphi_q]\,|\,q\in S^p_{\mapsto}\}$. Since $S^p_{\mapsto}$ is compact there is a finite subcover; let $\varphi(x,y)$ be the  disjunction of all the $\varphi_q$'s from the subcover. Note that $\varphi$ witnesses $p$-semi-isolation and that  $S^p_{\mapsto}\subseteq [\varphi]\subseteq S^p_{\rightarrow}$ holds. Define $x\leq y$ to be:
$$   x=y\vee\left(\varphi(x,y)\wedge(\forall t)(\varphi(y,t)\Rightarrow\varphi(x,t))\right)$$
Clearly, $\leq$ defines a quasi-order on $M$;  $[\varphi]\subseteq S^p_{\rightarrow}$ implies that $\leq$ witnesses $p$-semi-isolation.

\setcounter{clm}{0}
\begin{clm}
If $a\longmapsto b$ realize $p$ then $\varphi(b,M)\subsetneq \varphi(a,M)$ and $a<b$.
\end{clm}
\begin{littleproof}  Suppose that $d\in \varphi(b,M)$. Then $a\longmapsto b\longrightarrow d$ implies $a\longmapsto d$ and $\tp(ad)\in S^p_{\mapsto}\subseteq [\varphi]$. Thus $d\in\varphi(a,M)$  and $\varphi(b,M)\subsetneq \varphi(a,M)$ holds. Similarly  $a\longmapsto b$ implies $\tp(ad)\in S^p_{\mapsto}\subseteq [\varphi]$ so $\models\varphi(a,b)$. Finally, $\models\varphi(a,b)$ and $\varphi(b,M)\subsetneq \varphi(a,M)$ imply $a<b$.
\end{littleproof}

\medskip
Since $p$ is asymmetric no element of $p$ is maximal in the semi-isolation quasi-order. Then, by the claim,  no realization of $p$ is $\leq$-maximal. We conclude that $\leq$ defines a $p$-order on $p(M)$, proving condition (i) from the definition of partial semi-isolation. To prove (ii), suppose that
$a\leqmap b\longrightarrow c$ holds.  Then $a\longmapsto c$ and the claim implies
$a< c$. Therefore    $a\leqmap c$ holds,   proving (ii). $\leq$ partially defines semi-isolation on $p$.
\end{proof}

\begin{cor}\label{P11}
Suppose that   $p(x)\in S(\emptyset)$ is asymmetric and   that
$\overline{SI}_p$ is a relatively definable subset of $p(M)^2$.
Then semi-isolation is partially definable on $p(M)$. In
particular, $T$ has $SOP$.
\end{cor}
\begin{proof}Suppose that  $\overline{SI}_p$ is relatively definable within $p(M)^2$ and we will show that $S^p_{\mapsto}$ is closed in $S_{p,p}$. By Remark \ref{RS}(8) $S^p_{\rightarrow}\cup S^p_{\leftarrow}$ is closed; clearly it contains $S^p_{\mapsto}$ so $\cl(S^p_{\mapsto})\subseteq S^p_{\rightarrow}\cup S^p_{\leftarrow}$. On the other hand, by Remark \ref{RS}(9) we have   $\cl(S^p_{\mapsto})\subseteq S^p_{\mapsto}\cup S^p_{\perp}$. Therefore:
$$\cl(S^p_{\mapsto})\subseteq (S^p_{\rightarrow}\cup S^p_{\leftarrow})\cap(S^p_{\mapsto}\cup S^p_{\perp})=S^p_{\mapsto}\ .$$
Therefore $S^p_{\mapsto}$ is closed in $S_{p,p}$ and the conclusion follows  by Theorem \ref{Tpd}.
\end{proof}

\begin{cor}\label{Tacc} ($T$ is NSOP) If   $p\in S(\emptyset)$ is asymmetric then
$S^p_{\mapsto}$ (is infinite and) has an accumulation point in $S^p_{\perp}$. In particular, $p(x)\cup p(y)\cup x\bot^p y$ is  consistent.
\end{cor}
\begin{proof} By Remark \ref{RS}(9) we have   $\cl(S^p_{\mapsto})\subseteq S^p_{\mapsto}\cup S^p_{\perp}$.
The NSOP assumption combined with Theorem \ref{Tpd} implies that $S^p_{\mapsto}$ is not closed in $S_{p,p}$, so there exists   $q\in\cl(S^p_{\mapsto})\smallsetminus S^p_{\mapsto}$. Then $q$ is an accumulation point of $S^p_{\mapsto}$ and $q\in S^p_{\perp}$.
In particular, $S^p_{\perp}\neq\emptyset$  so $p(x)\cup p(y)\cup x\bot^p y$ is  consistent.
\end{proof}

Theories with few links were introduced by Benda in \cite{B}: $T$
has \emph{few links} if whenever $p(\bar x)$ and $q(\bar y)$ are
complete types then there are only finitely many complete types
$r(\bar x,\bar y)\supset p(\bar x)\cup q(\bar y)$ such that
$r(\bar c,\bar y)$ is nonisolated in $S(\bar c)$ for all $\bar c$
realizing $p(\bar x)$. Pillay in \cite{P1}  proved that any
Ehrenfeucht theory with few links has SOP. He noted that his proof
uses only the assumption when $p=q$ is a powerful type. Indeed, it
is not hard to realize  that the few links assumption implies that
$S^p_{\mapsto}$ is finite for any $p\in S(\emptyset)$: If $\bar
a,\bar b\models p$ and $\bar a\longmapsto \bar b$ then $\tp(\bar
a/\bar b)$ is nonisolated; there are only finitely many
possibilities for $\tp(\bar a/\bar b)$  so   $S^p_{\mapsto}$ is
finite. In particular, $S^p_{\mapsto}$ is closed in $S_{p,p}$  and
we have:

\begin{cor}  Any theory with few links and an asymmetric type   has the strict order property.
\end{cor}

In the same article Pillay   commented at the beginning of Section
3 the few links assumption:  ".. This condition is admittedly
rather artificial,  but it enables some proofs to go through ..."
An easy  consequence of the few links assumption is that
$CB(S_{p,p})\leq 1$ holds for all  $p\in S(\emptyset)$ (simply
because $S_{p,p}$ cannot have infinitely many accumulation
points). So $CB(S_{p,p})= 1$ seems to be a more natural condition.
There are such Ehrenfeucht theories, the first  example was found
by Woodrow in \cite{W1}.

\begin{que}Is there a powerful type $p$ in an NSOP theory satisfying $CB(S_{p,p})= 1$?
\end{que}

In this article we do not give much evidence towards answering this question.


\begin{cor}  ($T$ is small, NSOP) Suppose that  $p\in S(\emptyset)$ is asymmetric (not necessarily powerful) and that $CB(S_{p,p})= 1$ holds. Then:

\smallskip (1) $|S^p_{\mapsto}|\geq \aleph_0$  and   $|S^p_{\perp}|\geq 1$.

\smallskip (2) There are infinitely many pairwise inequivalent $p$-principal formulae.
\end{cor}
\begin{proof} (1) follows from  Corollary \ref{Tacc}. To prove (2) note that $CB(S_{p,p})= 1$ implies that there are infinitely many members of  $S^p_{\mapsto}$ isolated in  $S_{p,p}$. If  $\tp(ab)\in S^p_{\mapsto}$  is such a type then $\tp(b/a)$ is isolated and contains a $p$-principal formula.
\end{proof}






\section{Incomparability}\label{S3}

The next theorem   deals
with the case when $\overline{SI}_p$ has relatively definable
intersection with the product of two relatively definable subsets
of $p(M)$. The intended combinatorial description of this
situation is formalized in Proposition \ref{Pnis}: if we have two
large, unbounded relatively definable subsets of $p(M)$ then some
pair of their elements is incomparable.

\begin{thm}\label{P333}Suppose that $p\in S_1(\emptyset)$ is nonisolated and
that  $D_1, D_2\subset M$ are $\bar e$-definable   subsets of $M$
such that the following   conditions are satisfied:

\smallskip
 1)   $\overline{SI}_p\cap (D_1\times D_2)\neq \emptyset$ is relatively $\bar e$-definable
within $D_1\times D_2$;

\smallskip
 2)   For all $a\in D_1\cap p(M)$   there is $b\in D_{2}\cap p(M)$ such
that $a\longmapsto b$.

\smallskip
 3) For all $b\in D_2\cap p(M)$ there is $a\in D_1\cap p(M)$ such that
$b\longrightarrow a$.

\smallskip\noindent Then there is an $\bar e$-definable quasi-order on $M$ such that
no element of $D_1\cap p(M)$ is below a maximal one  of $D_1$. In
particular $T$ has the strict order property.
\end{thm}
\begin{proof}Suppose that   $D_i$ is defined
by $D_i(x,\bar e)$ and that relative definability is witnessed by
$\theta(x,y,\bar e)$. So we have:
$$p(x)\cup p(y)\cup\{D_1(x,\bar e),D_2(y,\bar
e),\theta(x,y,\bar e)\}\vdash y\in\Sem_p(x) \ \vee  \
x\in\Sem_p(y).$$ The right side is a long disjunction so, by
compactness, there  is an $L$-formula $\phi(x,y)$ witnessing
$y\in\Sem_p(x)$   and there is an $L$-formula $\psi(x,y)$
witnessing $x\in\Sem_p(y)$ such that:
$$p(x)\cup p(y)\cup\{D_1(x,\bar e),D_2(y,\bar e),\theta(x,y,\bar e)\}
\vdash\phi(x,y)\vee\psi(y,x).$$ Hence for any pair $(a,b)\in
D_1\times D_2$ of realizations of $p$ we have
\setcounter{equation}{0}
\begin{equation}\label{ec}\textmd{either} \
 \models\neg\theta(a,b,\bar e) \textmd{ \ \   or: \     at least one of $a\phiar b$ and
$b\psiar a$ holds}\end{equation}
The first disjunction here is exclusive because $\theta(x,y,\bar e)$ relatively defines $\overline{SI}_p\cap D_1\times D_2$.
Further we express assumption 3) by:
\begin{equation}\label{ec22}p(x) \cup\{D_2(x,\bar e) \}\vdash
\bigvee_{\psi'(x,y)} \exists y(D_1(y,\bar e)\wedge \psi'(x,y))
\end{equation}  where the disjunction is taken over all
$\psi'(x,y)$ witnessing $p$-semi-isolation. By compactness for some  $\psi'(x,y)$ we have:
\begin{equation}\label{eec}\textmd{for all }
 b\in D_2\cap p(M)   \textmd{  there is $c\in D_1\cap p(M)$ such that
$b\psiprimar c$ holds.}\end{equation}
After replacing both $\psi$ and $\psi'$ by their
disjunction, we may assume   $\psi=\psi'$. Let $\varphi(x,y,\bar
e)$ be \ $\exists z(D_2(z,\bar e)\wedge
\phi(x,z)\wedge\psi(z,y))$.\ Clearly, $\varphi(a,y,\bar e)$ forces
$p(y)$ for any $a$ realizing $p$.

\setcounter{clm}{0}
\begin{clm}\label{221}  For any $a\in D_1\cap p(M)$ there is $c\in D_1$
satisfying  $a\longmapsto c$ and $\models\varphi(a,c,\bar e)$.
\end{clm}
\begin{littleproof} Let  $a\in D_1\cap p(M)$. By 2) there is $b\in D_2\cap p(M)$ and by   (\ref{eec})  there is $c\in D_1\cap p(M)$ such that  $a\longmapsto b\psiar c$ holds. Then
$(a,b)\in \overline{SI}_p$ implies $\models\theta(a,b,\bar e)$,
and $a\notin\Sem_p(b)$ implies that $b\psiar a$ does not hold. By
(1) we derive   $a\phimap b$.   Thus $a\phimap b\psiar c$ and so
$\models\varphi(a,c,\bar e)$.
\end{littleproof}

\smallskip Define  \ \ $a'\leq b'$ \  iff \   $\varphi(b',M,\bar e)\cap D_1\subseteq
\varphi(a',M,\bar e)\cap D_1$.  \ Clearly, $\leq$  is a definable
quasi-order on $M$. We will show that no   element of $D_1\cap
p(M)$ is below a maximal one of $D_1$.

\begin{clm}\label{222}  If  $a,c \in D_1 \cap p(M)$ and $a\longmapsto c$
then $a\leq c$.
\end{clm}
\begin{littleproof} Suppose that  $d\in \varphi(c,M,\bar e)\cap D_1$.
Then  there is $b\in D_2$ such that $c\phiar b\psiar d$. Now,
$a\longmapsto c\longrightarrow b$ implies $a\longmapsto b$, so
$b\psiar a$ does not hold; also, $(a,b)\in \overline{SI}_p$
implies $\models\theta(a,b,\bar e)$.  By (1) we conclude that
$a\phimap b$ holds and then $a\phimap b\psiar d$ implies
$\varphi(a,d,\bar e)$. Thus $d\in \varphi(a,M,\bar e)$. This shows
that $\varphi(c,M,\bar e)\cap D_1\subseteq \varphi(a,M,\bar e)\cap
D_1$, i.e $a\leq c$.
\end{littleproof}

\smallskip
Now, let   $a_1\in D_1\cap p(M)$.  By  Claim \ref{221} there is
$c_1\in D_1$ such that $a_1\longmapsto c_1$ and
$\models\varphi(a_1,c_1,\bar e)$.   By Claim \ref{222} we have
$a_1\leq c_1$. Repeating the same procedure with $c_1$ we find
$a_2\in D_1$ satisfying: $c_1\longmapsto a_2$,
$\models\varphi(c_1,a_2,\bar e)$ and $c_1\leq a_2$. In particular
$a_1\leq a_2$, i.e.   $\varphi(a_2,M,\bar e)\cap D_1\subseteq
\varphi(a_1,M,\bar e)\cap D_1$. Then  $c_1\notin
\varphi(a_2,M,\bar e)$: otherwise  $\models\varphi(a_2,c_1,\bar
e)$ would witness  $a_2\longrightarrow c_1$ which is in
contradiction with $c_1\longmapsto a_2$.    Thus $c_1\in
\varphi(a_1,M,\bar e)\setminus \varphi(a_2,M,\bar e)$ and
$a_1<a_2$. Continuing in this way we get an infinite strictly
increasing chain of elements of $D_1\cap p(M)$.
\end{proof}

\section{Semi-isolation on minimal powerful types}\label{S4}

Throughout this section we will assume that $T$ (is small
and) has a powerful type. We will
say that $p\in S(\emptyset)$ is  a \emph{minimal powerful} type if
it is powerful and there is a formula $\theta(x)\in p$ such that
$p$ is the unique powerful type containing $\theta$. Minimal
powerful types exist in any Ehrenfeucht theory: take a powerful
type of minimal $CB$-rank.  To simplify notation, unless otherwise stated  we will
assume that $p\in S_1(\emptyset)$ is   powerful.

\smallskip We will be interested in  sets definable over a  single
parameter, for which  we do not a priori assume to realizes even a non-isolated type.   We will say that $D=\phi(d,M)$
is a \emph{$p$-set} if $D\cap p(M)\neq\emptyset$ and there exists $b\in D\cap p(M)$   such that at least one of the following two conditions hold:
\begin{enumerate}
\item $b$ does not semi-isolate $d$;

\item $\tp(d)$ is not powerful.
\end{enumerate}The
intended intuitive description of a $p$-set  is that $D\cap p(M)$
is large and unbounded; this is formalized in Lemma \ref{Lbd}
below.

\smallskip

\begin{rmk}\label{Rpset} Suppose that $p$ is a powerful type.

\smallskip

(1) If $\tp(d)$ is not powerful then the second condition from the
definition of a $p$-set is satisfied, so  $D=\phi(d,M)$ is a
$p$-set if and only if it contains a realization of $p$.

\smallskip (2)  Suppose that $p$ is a minimal powerful type  and that
$\theta(x)\in p$ witnesses the minimality. Let $d\in\theta(M)\smallsetminus p(M)$. Then, by part (1),  $D=\phi(d,M)$ is
a $p$-set whenever it contains a realization of $p$.

\smallskip (3) Suppose that $d\models p$ and that $\phi(x,y)$
witnesses the asymmetry of  $p$-semi-isolation: there are $a,b\in
p(M)$ such that \ $a\phimap b$.\, Then $b$ witnesses that the
first condition  from the definition holds   for $D=\phi(a,M)$, so
$\phi(a,M)$ is a $p$-set. In particular $\psi(a,M)$ is a $p$-set
for any $p$-principal formula  $\psi(x,y)$  and  $a\models p$.

\smallskip
(4) Suppose that $p$ is a minimal powerful type  and that the
minimality is witnessed by $\theta(x)\in p(x)$.   If $\phi(x,y)$
is  a $p$-principal formula, then for all $d\in \theta(M)$:
$D=\phi(d,M)$ is a $p$-set if and only if it contains a
realization of $p$. For $d\in p(M)$ this follows from part (3), and for $d\notin p(M)$ from part (1).
\end{rmk}

\begin{lem}\label{Lbd} Suppose that: $\theta(x)\in p(x)$ witnesses that $p\in S_1(\emptyset)$ is
a minimal powerful type,   $d\in \theta(M)$, and that
$D=\phi(d,M)$ is a $p$-set. Then $D\cap p(M)$ does not have an
$SI_p$-upper bound.
\end{lem}
\begin{proof}Suppose, on the contrary, that $a\in p(M)$ is an upper bound for $D\cap p(M)$. Then
$c\longrightarrow a$ holds for all $c\in D\cap p(M)$:
\[p(x)\cup\{\phi(d,x)\}\vdash \bigvee_{\psi}\psi(x,a)\] By
compactness there are $\theta_0(x)\in p(x)$ (wlog implying
$\theta(x)$) and $\psi(x,y)$ witnessing $p$-semi-isolation such
that \ $\models (\theta_0(x)\wedge \phi(d,x))\Rightarrow
\psi(x,a)$. Define: $$\sigma(y,z):=\forall t ((\theta_0(t)\wedge
\phi(y,t))\Rightarrow \psi(t,z))$$ Then $\models\sigma(d,a)$ holds
and, according to the definition we have two cases:

\smallskip \noindent {\bf Case 1.} \ There exists  $b\in D\cap p(M)$  such
that $b$ does not semi-isolate $d$. \

\smallskip\noindent
In this case we have: \setcounter{equation}{0}
\begin{equation}\label{equ2}\models \phi(d,b)\wedge \theta(d)\wedge\exists
z\sigma(d,z);\end{equation} Since $b$ does not semi-isolate $d$
any formula from $\tp(d/b)$ is consistent with   infinitely many
types from $S_1(\emptyset)$,  so there exists $d'\in M$ which does
not realize $p$ and satisfies (\ref{equ2})   in place of $d$. Note
that $\models\theta(d')$ and the minimality of $p$ together imply
that $\tp(d')$ is not powerful. Let $a'$ be such that:
$$\models \phi(d',b)\wedge \theta(d')\wedge\sigma(d',a')$$
We {\em claim}  that  $\sigma(d',z)\vdash p(z)$ holds. Assume
$\models\sigma(d',c)$. Then from  $b\in\theta_0(M)\cap \phi(d',M)$
and  the definition of $\sigma$  we get $\models\psi(b,c)$. Since
$\psi$ witnesses $p$-semi-isolation the claim follows.

$T$ is small, so there is an isolated type in $S_1(d')$ containing
$\sigma(d',t)$, it is an extension of $p$. Thus  $d'$ isolates an
extension of $p$ and, because $p$ is powerful, $\tp(d')$ has to be
powerful, too. A contradiction.

\smallskip
\smallskip \noindent {\bf Case 2.} \ $\tp(d)$ is not powerful.

\smallskip\noindent
Since $D$ is a $p$-set there exists $b'\in \phi(d,M)\cap p(M)$.
Assuming $\models \sigma(d,c')$ and arguing as in the first case
we derive $b'\psiar c'$ so  $\sigma(d,z)\vdash p(z)$.  Again we
can find an isolated extension of $p$ in $S_1(d)$ and conclude
that $\tp(d)$ is powerful. A contradiction.
\end{proof}

Next we show that $SI_p$-incomparability appears quite often on
the locus of a minimal powerful type in  an NSOP theory.

\begin{prop}\label{Pnis} ($T$ is NSOP) Suppose that: $\theta(x)\in p(x)$ witnesses that $p$ is
a minimal powerful type,   $d_i\in \theta(M)$,   and  that each
$D_i=\phi_i(d_i,M)$ is a $p$-set for $i=1,2$. Then there are
$a\in D_1, b\in D_2$ realizing $p$   such that $a\perp_p b$.
\end{prop}
\begin{proof}
Otherwise,  for all $a\in D_1, b\in D_2$    realizing $p$ we have
$(a,b)\in \overline{SI}_p$ so: \setcounter{equation}{0}
\begin{equation}\textmd{ at least one of $a\longrightarrow b$ and
$b\longrightarrow a$ holds.}\end{equation} In particular,
$\overline{SI}_p\cap (D_1\times D_2)$ is relatively
$d_1d_2$-definable within $p(M)^2$ and the first assumption  of
Theorem \ref{P333} is satisfied. We will prove that the other two
are satisfied, too.

Suppose that  the second condition fails and witness the failure
by $a\in D_1\cap p(M)$.  Then, by (1), $b\longrightarrow a$ would
hold for all $b\in D_2\cap p(M)$, so $a$ would be an upper bound for
$D_2\cap p(M)$; this is in contradiction with Lemma \ref{Lbd}.
Therefore the second and, similarly, the third condition are
fulfilled. By Theorem \ref{P333}  $T$ has the strict order
property. A contradiction.
\end{proof}

Thus   $SI_p$ is in some sense a "wide" quasi order. Because $p$ is powerful, it is also directed downwards. It is
interesting to know whether it has to be directed upwards.

\begin{que}
Must $SI_p$ be directed upwards on the locus of a minimal powerful
type in an NSOP theory?
\end{que}

We have proved in Corollary \ref{Tacc} that  $S^p_{\perp}\neq \emptyset$ and here, under much stronger assumptions,  we will prove that $|S^p_{\perp}|\geq 2$.

\begin{prop}
Suppose that $T$ is a  binary NSOP theory with  3 countable models
and that $p\in S_1(\emptyset)$ has $CB$-rank 1. Then \
 $q(x,y)=p(x)\cup p(y)\cup   x\bot_p y $ \  has at least two completions
in $S_2(\emptyset)$.
\end{prop}
\begin{proof}In a theory with 3 countable  models  there is a unique isomorphism type of a "middle model",
i.e a  countable model prime over a realization of a nonisolated
type.  the middle model    is weakly saturated because every
finitary type is realized in some finitely generated model. Thus
any nonisolated type is powerful and, in particular, $p$ is
powerful. Let $\theta(x)\in p$ be a formula of $CB$-rank 1 and
$CB$-degree 1. Then $p$ is the unique nonisolated type containing
$\theta(x)$ and $p$ is a minimal powerful type.

 $p$ is asymmetric so, by Corollary
\ref{Tacc},  $q(x,y)$ is consistent. Now suppose that the conclusion of the proposition fails:  $q(x,y)$ has
a unique completion  $q'(x,y)\in S_2(\emptyset)$. Choose  $a\,b\models q'$, then  $a\perp_p b$ holds.  By Corollary \ref{Tacc} $q'$ is an accumulation point of $S^p_{\mapsto}$ so each of  $\tp(ab)$, $\tp(a/b)$ and $\tp(b/a)$ is nonisolated.  By the three model assumption, we know that the model prime over $ab$
is also prime over a realization $d$ of $p$ (because any two models
prime over a realization of a nonisolated type are isomorphic).
Note that both $\tp(ab/d)$ and $\tp(d/ab)$ are isolated. Hence
there is a formula $\tau(x,y,z)\in \tp(dab)$  such that $\tau(d,y,z)$ isolates $\tp_{yz}(ab/d)$ and $\tau(x,a,b)$ isolates $\tp_x(d/ab)$.
Now we use the assumption that $T$ is binary: there are formulas $\phi',\psi',\sigma$ such that \ $$\models  (\phi'(x,y)\wedge \psi'(x,z)\wedge \sigma(y,z))\leftrightarrow \tau(x,y,z)\ .$$
The assumed isolation properties of $\tau$ imply:
\setcounter{equation}{0}
\begin{equation}\label{e1260}
  \phi'(x,a)\wedge\psi'(x,b)\wedge\sigma(a,b)  \vdash
p(x);\end{equation}
\begin{equation}\label{e1270}
  \phi'(d,y)\wedge\psi'(d,z)\wedge\sigma(y,z)  \vdash
\tp(ab/d).\end{equation} Let   $\tp(a/d)$ be isolated by
$\phi(d,y)\in \tp(a/d)$ and let $\tp(b/d)$ be isolated by
$\psi(d,z)\in\tp(b/d)$. Without loss of generality assume that
they are chosen so that $\models (\phi(x,y)\Rightarrow
\phi'(x,y))\wedge (\psi(x,y)\Rightarrow \psi'(x,y))$.
Then by (\ref{e1260}) and (\ref{e1270}):
\begin{equation}\label{e126}
  \phi(x,a)\wedge\psi(x,b)\wedge\sigma(a,b)  \vdash
p(x);\end{equation}
\begin{equation}\label{e127}
  \phi(d,y)\wedge\psi(d,z)\wedge\sigma(y,z)  \vdash
\tp(ab/d).\end{equation} Now consider the formula  \ $(\exists
x)(\theta(x)\wedge \phi(x,y)\wedge\psi(x,z)\wedge\sigma(y,z))$  \
which is in $\tp_{yz}(ab)=q'(y,z)$. Since $S^p_{\perp}=\{q'\}$, by
Corollary \ref{Tacc},  $q'$ is an accumulation point of
$S^p_{\mapsto}$, so there are $a'b'$ satisfying this formula such
that $\tp(a'b')\in S_{\mapsto}$; hence $(a',b')\in SI_p$. Then for
some $d'$ we have:
\begin{equation}\label{e128}
 \models \theta(d')\wedge\phi(d',a')\wedge\psi(d',b')\wedge\sigma(a',b') \, .   \end{equation}
$d'$ does not realize $p$: otherwise (\ref{e127}) would imply
$a'b'\models q'$ which is in contradiction with $(a',b')\in SI_p$.
\  Thus  $d'\in\theta(M)\smallsetminus p(M)$ so, by Remark
\ref{Rpset}(2), $D_1=\phi(d',M)$ and $D_2=\psi(d',M)$ are
$p$-sets. By Proposition \ref{Pnis} there are $a''\in D_1$ and
$b''\in D_2$ realizing $p$ such that $a''\perp_p b''$ holds. The
uniqueness of $q'$ implies $a''b''\models q'$ and  $\models
\sigma(a'',b'')$. Thus
$$\models    \phi(d',a'')\wedge\psi(d',b'')\wedge\sigma(a'',b'')$$
By (\ref{e126}) and  $\tp(ab)=\tp(a''b'')=q'$ we get $d'\models p$. A contradiction.
\end{proof}

\section{PGPIP for binary theories}\label{S5}

Throughout this section    we will assume that $T$ is a small,
binary theory and that $p$ is  a powerful 1-type. We have already
noted in Remark \ref{R111}  that $SI_p$ is directed downwards. In
Remark \ref{Rnew2} we noted a stronger form:  for any pair of
elements $a,b\in p(M)$ there exists $d\in p(M)$ and $p$-principal
formulas $\phi,\psi$ such that both $d\phiar a$ and $d\psiar b$
hold. In all  the basic examples $\phi$ and $\psi$ can be chosen
from a finite (fixed in advance) set. This property is labelled in
\cite{S2} as the global pairwise intersection property for $p$
(GPIP). Precisely, it means that there is a formula $\phi(x,y)$
which is a disjunction of $p$-principal formulae and such that
$(p(M),\phi(M^2))$ is an acyclic digraph satisfying:
\setcounter{equation}{0}\begin{equation} \textmd{for all $a,b\in
p(M)$ there exists $d\models p$ such that $\models
\phi(d,a)\wedge\phi(d,b)$.}\end{equation} Here we introduce  a bit
stronger property.

\begin{dfn} $p$ has PGPIP if there is a formula $\phi(x,y)$ which
is a disjunction of $p$-principal formulae and is  such that:
$(p(M)^2,\phi(M))$ is an acyclic digraph and for all $a,b\in p(M)$
there exists $d\models p$ satisfying:
\begin{equation} \textmd{ $\tp(ab/d)$ is
isolated and $\models \phi(d,a)\wedge\phi(d,b)$.}\end{equation}
\end{dfn}

We leave to the reader to check that nonisolated 1-types from the
Ehrenfeucht's and  Peretyatkin's (see \cite{Per}) examples have
PGPIP.

\begin{thm}($T$ is binary, NSOP)  Suppose that
$\phi(x,y)=\bigvee_{i=1}^n\phi_i(x,y)$, where each $\phi_i(x,y)$ is
$p$-principal, witnesses PGPIP for $p$. Then $n\geq 2$ and
 $CB(S_{p,p}(\emptyset))<  n^2$.
\end{thm}
\begin{proof}Fix $d$ realizing $p$. For each pair $i,j\leq n$
define:
$$D_{(i,j)}=\{(a,b)\in p(M)^2\,|\,\tp(ab/d) \textmd{ is isolated
and } \models \phi_i(d,a)\wedge\phi_j(d,b)\}$$
$$C_{(i,j)}=\{\tp(ab/d)\,|\, (a,b)\in D_{(i,j)}\} \ \ \
S_{(i,j)}=\{\tp(ab)\,|\, (a,b)\in D_{(i,j)}\} $$ Note that PGPIP
implies that\ $\bigcup_{(i,j)}S_{(i,j)}=S_{p,p}(\emptyset)$ \
holds; in particular, if $n=1$ then
$S_{(1,1)}=S_{p,p}(\emptyset)$.

\setcounter{clm}{0}\begin{clm} For every $q(x,y)\in S_{(i,j)}$
there is $\theta_q(x,y)\in q$   which has a unique extension in
$C_{(i,j)}$.
\end{clm}
\begin{littleproof} Let $(a,b)\in D_{(i,j)}$ realize $q$. Then $\tp(ab/d)$ is isolated and,
because  $T$ is binary and $\phi_i$'s are $p$-principal, there is
a formula $\theta_q(x,y)\in q(x,y)$ such that:
$$\phi_i(d,x)\wedge\phi_j(d,y)\wedge\theta_q(x,y)\vdash \tp(ab/d)$$
Since any extension of $\theta_q(x,y)$ in $C_{(i,j)}$ contains the
formula on the left hand side, we conclude that the extension is
unique.
\end{littleproof}

Now, we claim that each $S_{(i,j)}$ is a discrete subset of
$S_{p,p}(\emptyset)$.  Suppose, on the contrary, that $q(x,y)\in
S_{(i,j)}$ is an  accumulation point of $S_{(i,j)}$.  Then
$\theta_q$ is contained in  some $q'\in S_{(i,j)}$ which is
distinct from $q$. Thus $\theta_q$ has at least two extensions in
$C_{(i,j)}$: the one extending $q$ and the one extending $q'$. A
contradiction.

\smallskip
The first part of our theorem follows: if $n=1$ then
$S_{(1,1)}=S_{p,p}(\emptyset)$ is discrete and, because it is
compact,   it has to be finite.  Then  by Corollary \ref{Cfinite},
$T$ has SOP. A contradiction. Therefore $n\geq 2$.

\smallskip The second part follows from the following topological fact: A
compact space which is a union of $m$ discrete subsets has
$CB$-rank smaller than  $m$ (easily proved by induction). In our
situation $S_{p,p}(\emptyset)=\bigcup_{(i,j)}S_{(i,j)}$ is a union
of $n^2$ discrete subsets, so $CB(S_{p,p}(\emptyset))< n^2$.
\end{proof}

\end{document}